\documentclass[letterpaper, 10pt, conference]{ieeeconf}

\IEEEoverridecommandlockouts
\usepackage[colorlinks=true, pdfstartview=FitV, linkcolor=black, citecolor= black, urlcolor= black]{hyperref}
\usepackage[english]{babel}
\usepackage{graphicx,subcaption}
\usepackage{amsmath,amssymb,mathrsfs}
\usepackage{bm}
\usepackage{dsfont}% for mathbb 
\usepackage{physics}
\usepackage{xr}
\usepackage{algorithm,algpseudocodex}
\usepackage{lipsum}
\usepackage{centernot}
\usepackage{balance}
\usepackage{mathtools}
\usepackage{cite}
\usepackage{textcomp}
\usepackage{xcolor}

\usepackage{tikz}

\newcommand\submittedtext{%
  \footnotesize 
  © 2025 IEEE.  Personal use of this material is permitted.  Permission from IEEE must be obtained for all other uses, in any current or future media, including reprinting/republishing this material for advertising or promotional purposes, creating new collective works, for resale or redistribution to servers or lists, or reuse of any copyrighted component of this work in other works.
  Accepted for the 2025 IEEE Conference on Decision and Control (CDC).}

\newcommand\submittednotice{%
\begin{tikzpicture}[remember picture,overlay]
\node[anchor=south,yshift=10pt] at (current page.south) {\fbox{\parbox{\dimexpr1.0\textwidth-\fboxsep-\fboxrule\relax}{\submittedtext}}};
\end{tikzpicture}%
}

\newcommand{\R}{\mathbb{R}}

\newcommand{\Z}{\mathbb{Z}}
\newcommand{\PSD}{\mathbb{S}_+}
\newcommand{\PD}{\mathbb{S}_{++}}

\newcommand{\cov}[1]{\mathrm{Cov}[#1 ]}
\let\P\relax
\newcommand{\P}[1]{\mathbb{P}\left[#1\right]}
\let\E\relax
\newcommand{\E}[1]{\mathbb{E}\left[#1\right]}

\usepackage{cleveref}
\crefname{align}{}{}
\crefname{equation}{}{}
\crefname{figure}{Fig.}{Figs.}
\crefname{table}{Table}{Tables}
\crefname{theorem}{Theorem}{Theorems}
\crefname{definition}{Definition}{Definitions}
\crefname{lemma}{Lemma}{Lemmas}
\crefname{remark}{Remark}{Remarks}
\crefname{assumption}{Assumption}{Assumptions}
\crefname{proof}{Proof}{Proofs}
\crefname{algorithm}{Algorithm}{Algorithms}
\crefname{problem}{Problem}{Problems}
\crefname{proposition}{Proposition}{Propositions}
\crefname{corollary}{Corollary}{Corollaries}
\crefname{section}{Section}{Sections}
\crefname{proof}{Proof}{Proofs}

\usepackage{amsthm}

\newtheorem{remark}{Remark}
\newtheorem{theorem}{Theorem}
\newtheorem{lemma}{Lemma}
\newtheorem{assumption}{Assumption}

\begin{document}

%  Set figure path
\graphicspath{{./figures/}}

\title{\LARGE \bf
Hands-Off Covariance Steering: Inducing Feedback Sparsity via Iteratively Reweighted $\ell_{1,p}$ Regularization
\thanks{This work was supported by the U.S. Air Force Office of Scientific Research through research grant FA9550-23-1-0512. 
N. Kumagai acknowledges support for his graduate studies from the Shigeta Education Fund. The authors are with the School of Aeronautics and Astronautics, Purdue University, West Lafayette, Indiana, 47907, USA. Emails: 
    nkumagai@purdue.edu,
    koguri@purdue.edu
}%
}%

% \author{\IEEEauthorblockN{Naoya Kumagai} \and \IEEEauthorblockN{Kenshiro Oguri}}%
\author{Naoya Kumagai and Kenshiro Oguri}%

\maketitle%
\thispagestyle{empty}
\pagestyle{empty}

\submittednotice

\begin{abstract}
    We consider the problem of optimally steering the state covariance matrix of a discrete-time linear stochastic system
    to a desired terminal covariance matrix, while inducing the control input to be zero over many time intervals.
    We propose to induce sparsity in the feedback gain matrices by using a sum-of-norms version of the iteratively reweighted $\ell_1$-norm minimization.
    We show that the lossless convexification property holds even with the regularization term.
    Numerical simulations show that the proposed method produces a Pareto front of transient cost and sparsity
    that is not achievable by a simple $\ell_1$-norm minimization and closely approximates the $\ell_0$-norm minimization
    obtained from brute-force search.
\end{abstract}

\section{Introduction}
Sparsity is important for  the control and identification of systems.
It is often represented by the $\ell_0$-(pseudo) norm, which is the number of nonzero elements in a vector.
For example, a sparse signal recovery problem can be posed as
\begin{equation} \label{eq:L0-norm_minimization}
    \min_x \ \|x\|_0 \quad \text{s.t.} \quad y = Ax
\end{equation}
where $x$ is the signal to be recovered, $y$ is the noisy measurement, and $A$ is the measurement matrix.
Unfortunately, the $\ell_0$-norm is non-convex and nondifferentiable, and \cref{eq:L0-norm_minimization} is NP-hard \cite{natarajanSparseApproximateSolutions1995}.
On the other hand, the $\ell_1$-norm is convex, and its sparsity-inducing property has been demonstrated in various applications
\cite{tibshiraniRegressionShrinkageSelection1996,chenAtomicDecompositionBasis2001,candesEnhancingSparsityReweighted2008,boydConvexOptimization2004}. 
In system identification, perhaps the most well-known algorithm for inducing sparsity is the LASSO algorithm \cite{tibshiraniRegressionShrinkageSelection1996}.
Compared to ordinary least squares, the LASSO algorithm adds a constraint on the $\ell_1$-norm of the parameter vector,
which induces sparsity in the model selection. Basis Pursuit \cite{chenAtomicDecompositionBasis2001}
uses the $\ell_1$-norm in the objective function for sparse signal recovery, solving the problem
\begin{equation}
    \min_x \ \|x\|_1 \quad \text{s.t.} \quad y = Ax .
\end{equation}
The applications of the $\ell_1$-norm regularization are vast, and we refer the reader to 
\cite{candesEnhancingSparsityReweighted2008} for a comprehensive review.

Sparsity can be especially favorable in control of systems where 
every control input incurs a large cost (e.g. financial/human), such as in multi-period investment problems \cite{boydPerformanceBoundsSuboptimal2013}, 
economics \cite{leonardOptimalControlTheory1992}, and on-ground computation of spacecraft trajectory correction maneuvers \cite{oguriChanceConstrainedControlSafe2024b}.
As pointed out in \cite{pakazadSparseControlUsing2013}, an optimal control problem with sparsity consideration can be considered as a multi-objective optimization problem 
that trades control performance and control input sparsity.
Ref. \cite{hassibiLowAuthorityControllerDesign1999a} solves an $\ell_1$-norm minimization problem to design sparse feedback gains, i.e.
only a few sensor-to-actuator paths are used for feedback control. 
Ref. \cite{ohlssonTrajectoryGenerationUsing2010} solves a sum-of-norms regularization problem to track waypoints
while also regularizing the sum-of-norms of the impulse train which constitute the input signal.
Sum-of-norms regularization is a generalization of the $\ell_1$-norm regularization, where the parameter vector is divided into groups,
and the sum of the $p$-norms of the groups is minimized, promoting \textit{group sparsity}.
Ref. \cite{nagaharaMaximumHandsOffControl2016} shows the equivalence between the sparsest control and $\ell_1$-optimal control.
Recent works consider both the transient cost and sparsity, for example, for linear quadratic regulator (LQR) problems.
Representative solution methods are the alternating direction method of multipliers (ADMM) \cite{jovanovicSparseQuadraticRegulator2013}, 
branch and bound \cite{gaoCardinalityConstrainedLinearQuadratic2011}, and $\ell_1$-norm regularization \cite{zhangLinearQuadraticTracking2021}.
Some works refer to control problems with sparsity consideration as hands-off control, which 
gives the name \textit{hands-off covariance steering} for our problem.

Covariance steering \cite{hotzCovarianceControlTheory1987,bakolasOptimalCovarianceControl2016,okamotoOptimalCovarianceControl2018,chenOptimalSteeringLinear2018,liuOptimalCovarianceSteering2024,kumagaiChanceConstrainedGaussianMixture2024}
is a problem of finding the optimal feedback gain matrices that minimize
the expected quadratic cost function, while steering the state covariance matrix to a desired terminal covariance matrix under linear noisy dynamics.
Combined with control sparsity, hands-off covariance steering can control the state distribution to satisfy distributional/chance constraints while
keeping the control input zero for many time intervals, relieving the implicit cost for implementing controls in real-world systems.

As the solution method, we return to the question posed by \cite{candesEnhancingSparsityReweighted2008} for sparse signal recovery: 
``Can we improve upon $\ell_1$-norm minimization to better approximate $\ell_0$-norm minimization?" 
The influential work \cite{candesEnhancingSparsityReweighted2008} proposes to
solve the iteratively reweighted $\ell_1$-norm minimization (IRL1)%
\begin{equation} \label{eq:IRL1}
    \min_x \ \sum_i w_i |x_i| \quad \text{s.t.} \quad y = Ax
\end{equation}%
with weights $w_i$ updated based on the solution of the previous iteration. 
The authors show that this recovers the original signal $x$ 
more accurately than the unweighted $\ell_1$-norm minimization.
While a common approach in $\ell_1$-norm regularization is to tune the scalar regularization parameter to achieve the desired sparsity,
we observed that for covariance steering, even extremely large regularization parameters failed to produce sufficiently sparse solutions.

Our main contributions are as follows:
\begin{itemize}
    \item We propose a method to induce sparse control in the chance-constrained covariance steering problem
    by combining the iteratively reweighted $\ell_{1}$-norm minimization with sum-of-norms regularization.
    \item We show that the lossless convexification property holds with the sparsity-promoting regularization term, i.e. 
    the originally nonconvex problem can be relaxed to a convex one (here, a semidefinite program) and the solution to the 
    relaxed problem is also a solution to the original problem. 
    This enables efficient computation via solving a series of semidefinite programming problems.
\end{itemize}
% The rest of the paper is organized as follows. 
% \cref{sec:preliminaries} provides a brief review of covariance steering, $\ell_1$-norm regularization, and its iteratively reweighted version.
% \cref{sec:sparse_covariance_feedback} introduces the hands-off covariance steering problem and its theoretical justification.
% \cref{sec:numerical_example} demonstrates the solution obtained with the proposed method.

\textit{Notation:} 
$\norm{\cdot}_0$, $\norm{\cdot}_1$, $\norm{\cdot}_2$, $\norm{\cdot}_F$, $\tr(\cdot)$, 
$\E{\cdot}$, $\cov{\cdot}$, $\P{\cdot}$, $\det(\cdot)$, $\lambda_{\max}(\cdot)$ denote 
the $\ell_0$, $\ell_1$, $\ell_2$, Frobenius norm, trace, expectation, covariance, probability, determinant, and maximum eigenvalue, respectively.
$\R, \R^n, \R^{n \times m}$ denote real numbers, $n$-dimensional real-valued vectors, and $n \times m$ real-valued matrices, respectively.
$\PSD^n$ ($\PD^n$) denotes the set of $n \times n$ symmetric positive semidefinite (definite) matrices.
$\Z_{a:b}$ denotes the set of integers from $a$ to $b$. 
The symbol $\succeq$ ($\succ$) denotes (strict) matrix inequality between symmetric matrices.
\section{Preliminaries} \label{sec:preliminaries}
In this section, we provide a brief review of covariance steering, $\ell_1$-norm regularization, and its reweighted version.
\subsection{Covariance Steering}
Covariance steering aims to design a feedback control policy that minimizes the expected cost function while steering the covariance matrix of the state to a desired terminal covariance matrix.
Recently, both the continuous-time \cite{chenOptimalSteeringLinear2018} and discrete-time \cite{bakolasOptimalCovarianceControl2016,okamotoOptimalCovarianceControl2018,liuOptimalCovarianceSteering2024} versions of the problem have been studied.
Here, we consider the finite-horizon, discrete-time version with horizon $N$:%
\begin{subequations} \label{prob:covariance_steering}
\begin{align}
    \min_\pi &\quad J_{\Sigma} := \E{\sum_{k=0}^{N-1} x_k^\top Q_k x_k + u_k^\top R_k u_k } \\ 
    \text{s.t.} \quad &x_{k+1} = A_k x_k + B_k u_k + D_k w_k, \quad  k \in \Z_{0:N-1} \label{eq:ltv_dynamics}\\
    &\E{x_0} = \mu_0, \ \E{x_N} = \mu_N\\
    &\cov{x_0} = \bar{\Sigma}_0, \ \cov{x_N} \preceq \bar{\Sigma}_N   \label{eq:boundary_conditions}\\
    &u_k = \pi(x_k, k), \quad  k \in \Z_{0:N-1}
\end{align}
\end{subequations}
where $x_k \in \R^n$, $u_k \in \R^m$ are the state and control input at time $k$, respectively.
$w_k \in \R^p$ is a white, zero-mean, Gaussian noise with identity covariance matrix.   
$A_k \in \R^{n \times n}$, $B_k \in \R^{n \times m}$, $D_k \in \R^{n \times p}$ are the system matrices.
$Q_k \in \PSD^{n}$, $R_k \in \PD^n$ are the state and control input cost matrices.
$\mu_0 \in \R^n$, $\mu_N \in \R^n$ are the initial and terminal state mean vectors.
$\bar{\Sigma}_0 \in \PD^n$, $\bar{\Sigma}_N \in \PD^{n}$ are the initial and terminal covariance matrices, with
$\bar{\Sigma}_N \succ D_{N-1} D_{N-1}^\top$.

\begin{assumption} \label{assum:invertible_state_transition}
    $A_k$ is invertible for $k \in \Z_{0:N-1}$.
\end{assumption}
This assumption is reasonable, since the discretized state transition matrix of a continuous-time dynamical system is invertible.

\begin{assumption} \label{assum:nondegenerate_covariance}
    $\cov{x_k} \succ 0$ for $k \in \Z_{0:N}$.
\end{assumption}
This assumption is also reasonable, as most real-life problems involve nondegenerate covariance matrices.

\begin{assumption} \label{assum:zero_mean}
    $\mu_0 = \mu_N = 0$,
$\E{x_k} = 0$ for $k \in \Z_{0:N}$.
\end{assumption}
This assumption is made for simplicity and allows us to focus on the covariance dynamics in this paper.

We consider an affine feedback policy, which has been shown to be optimal 
\cite{bakolasOptimalCovarianceControl2016,liuOptimalCovarianceSteering2024}
in the unconstrained covariance steering setting\footnote{The feedforward term which exists in \cite{liuOptimalCovarianceSteering2024}
is not included, due to \cref{assum:zero_mean}.}:%
\begin{equation} \label{eq:control_policy}
    u_k = K_k x_k
\end{equation}
where $K_k$ is the feedback gain matrix.
From standard derivations, the covariance propagation equation is given by
\begin{align} \label{eq:covariance_propagation_in_K}
    \Sigma_{k+1} &= (A_k + B_k K_k) \Sigma_k (A_k + B_k K_k)^\top + D_k D_k^\top
\end{align}
and the objective function by
\begin{equation}\label{eq:cost_function_in_K}
    J_{\Sigma} = \sum_{k=0}^{N-1} \tr(Q_k \Sigma_k) + \tr(R_k K_k \Sigma_k K_k^\top) .
\end{equation}
Then, \cref{prob:covariance_steering} is cast as a parameter optimization problem:
\begin{equation} \label{eq:covariance_steering_in_K}
    \min_{K_k,\Sigma_k} \ \cref{eq:cost_function_in_K} \quad 
    \text{s.t.} \ \cref{eq:covariance_propagation_in_K}, \ \Sigma_0 = \bar{\Sigma}_0, \ \Sigma_N \preceq \bar{\Sigma}_N
\end{equation}

Note that both the feedback gains $K_k$ and the covariance matrices $\Sigma_k$ are optimization variables,
making \cref{eq:covariance_steering_in_K} nonconvex due to the equality constraint \cref{eq:covariance_propagation_in_K}.
However, through a process called \textit{lossless convexification}, the solution to \cref{eq:covariance_steering_in_K} 
can be obtained by solving a convex optimization problem \cite{liuOptimalCovarianceSteering2024}.
Here, we review the solution method proposed in \cite{liuOptimalCovarianceSteering2024}.
We begin by performing a change of variables by defining
\begin{equation}
    U_k := K_k \Sigma_k .
\end{equation}
Note that $K_k$ can be recovered uniquely from $U_k$ as $K_k = U_k \Sigma_k^{-1}$ from \cref{assum:nondegenerate_covariance}.
With this change of variables, \cref{eq:covariance_propagation_in_K} and \cref{eq:cost_function_in_K} are respectively rewritten as
\begin{align} \label{eq:covariance_dynamics_Sigma_inv}
    \Sigma_{k+1} &= A_k \Sigma_k A_k^\top + A_k U_k^\top B_k^\top + \nonumber \\
    & \quad B_k U_k A_k^\top + B_k U_k \Sigma_k^{-1} U_k^\top B_k^\top + D_k D_k^\top \\ 
    J_{\Sigma} &= \sum_{k=0}^{N-1} \tr(Q_k \Sigma_k) + \tr(R_k U_k \Sigma_k^{-1} U_k^\top)     \label{eq:cost_Sigma_inv} .
\end{align}
Next, introduce variables 
\begin{equation} \label{eq:Y_definition}
    Y_k := U_k \Sigma_k^{-1} U_k^\top (= K_k \Sigma_k K_k^\top)
\end{equation}
and replace the relevant expressions in \cref{eq:covariance_dynamics_Sigma_inv,eq:cost_Sigma_inv}.
While \cref{eq:Y_definition} is nonconvex, it can be relaxed to a matrix inequality as
\begin{equation} \label{eq:Y_relaxation}
    Y_k \succeq U_k \Sigma_k^{-1} U_k^\top
\end{equation}
which, from Schur's Lemma, is equivalent to the linear matrix inequality 
\begin{equation}
    \begin{bmatrix}
        \Sigma_k & U_k ^\top \\
           U_k & Y_k
    \end{bmatrix} \succeq 0 .
\end{equation}
\cref{eq:covariance_steering_in_K} now takes the form of a semidefinite programming (SDP):%
\begin{subequations} \label{eq:covariance_steering_in_U}
\begin{align}
    &\min_{\Sigma_k, U_k, Y_k} \ J_{\Sigma} = \sum_{k=0}^{N-1} \tr(Q_k \Sigma_k) + \tr(R_k Y_k) \\
    &\quad \text{s.t.} \quad 
     \Sigma_0 = \bar{\Sigma}_0 , \quad
     \Sigma_N \preceq \bar{\Sigma}_N \label{eq:boundary_conditions_2}\\
    & \Sigma_{k+1} = A_k \Sigma_k A_k^\top + A_k U_k^\top B_k^\top + B_k U_k A_k^\top \nonumber \\
     & \quad + B_k Y_k B_k^\top + D_k D_k^\top, \quad k \in \Z_{0:N-1} \label{eq:covariance_dynamics_in_Y}\\
     & \begin{bmatrix}
         \Sigma_k & U_k ^\top \\
            U_k & Y_k
        \end{bmatrix} \succeq 0 , \quad k \in \Z_{0:N-1} \label{eq:psd_constraint}
\end{align}
\end{subequations}
\cite{liuOptimalCovarianceSteering2024} shows that the solution obtained 
from solving this convex problem (after changing back the variables to $K_k$) is
equivalent to the solution obtained from the original nonconvex problem \cref{eq:covariance_steering_in_K}.
Since the equality \cref{eq:Y_definition} is satisfied at the optimal solution of \cref{eq:covariance_steering_in_U}, 
$Y_k$ will be equivalent to the covariance matrix of the input $u_k$.

\subsection{Group Sparsity via Sum-of-Norms Regularization}
While $\ell_1$ regularization is useful for inducing element-wise sparsity,
it cannot be directly applied to induce sparsity in \textit{groups} of elements \cite{schmidtLectureNotesCPSC}.
In this case, the regularization term takes the general sum-of-norms form \cite{yuanModelSelectionEstimation2006}:
\begin{equation}
    \sum_{g \in G} \|x_g\|_p
\end{equation}
Here, $G$ is a set of groups, and $x_g$ is the vector of features in group $g$.
This regularization promotes that all coefficients in some of the groups are zero, i.e. it encourages
the use of only a few groups. Naturally, this is also called the 
$\ell_{1,p}$-norm and can be thought of as a generalization of the $\ell_1$-norm. 
There are several choices for $p$; $p=2$ is standard in group LASSO \cite{yuanModelSelectionEstimation2006},
$p=\infty$ is also sparsity-inducing. 
However, $p=1$ does not promote group sparsity \cite{schmidtLectureNotesCPSC}.

\subsection{Iteratively Reweighted $\ell_1$ Minimization}
Although the $\ell_1$-norm regularization is widely used for inducing sparsity, its solution can be suboptimal in terms of sparsity.
In the context of signal processing, \cite{candesEnhancingSparsityReweighted2008} proposes IRL1
 to better approximate the $\ell_0$-norm minimization, as in \cref{eq:IRL1}. By iteratively solving the weighted $\ell_1$-norm minimization problem
with weights updated based on the solution of the previous iteration, \cite{candesEnhancingSparsityReweighted2008} shows that
the solution recovers the original signal $x$ accurately. 
The weight update rule is given by
\begin{equation} \label{eq:weight_update_rule}
    w_i^{(l+1)} = \frac{1}{|x_i^{(l)}| + \epsilon}, \quad w_i^{(1)} = 1
\end{equation}
where $l$ is the iteration index, and $\epsilon$ is a small positive constant, 
used to avoid division by zero.
Intuitively, the weights are updated such that for elements of $x$ that are close to zero
in the previous iteration, the weights are large, promoting sparsity in the subsequent solution.
\cref{eq:weight_update_rule} is the first-order term in the linearization of the (nonconvex)
penalty function \cite{candesEnhancingSparsityReweighted2008}
\begin{equation}
    \phi_{\log,\epsilon}(t) = \log(1 + |t|/\epsilon)
\end{equation}
for $t = x_i^{(l)}$, which, when multipled by a constant $c = 1 / \phi_{\log,\epsilon}(1)$, approaches the discontinuous penalty function 
\begin{equation}
    \phi_0(t) = \begin{cases}
        1 & t \neq 0 \\
        0 & t = 0
    \end{cases}
\end{equation}
as $\epsilon \rightarrow 0$.
See \cref{fig:log_penalty_both} for a visualization.
\begin{figure}[htb]
    \centering
    \begin{subfigure}[c]{0.47\linewidth}
        \centering
        \includegraphics[width=\linewidth]{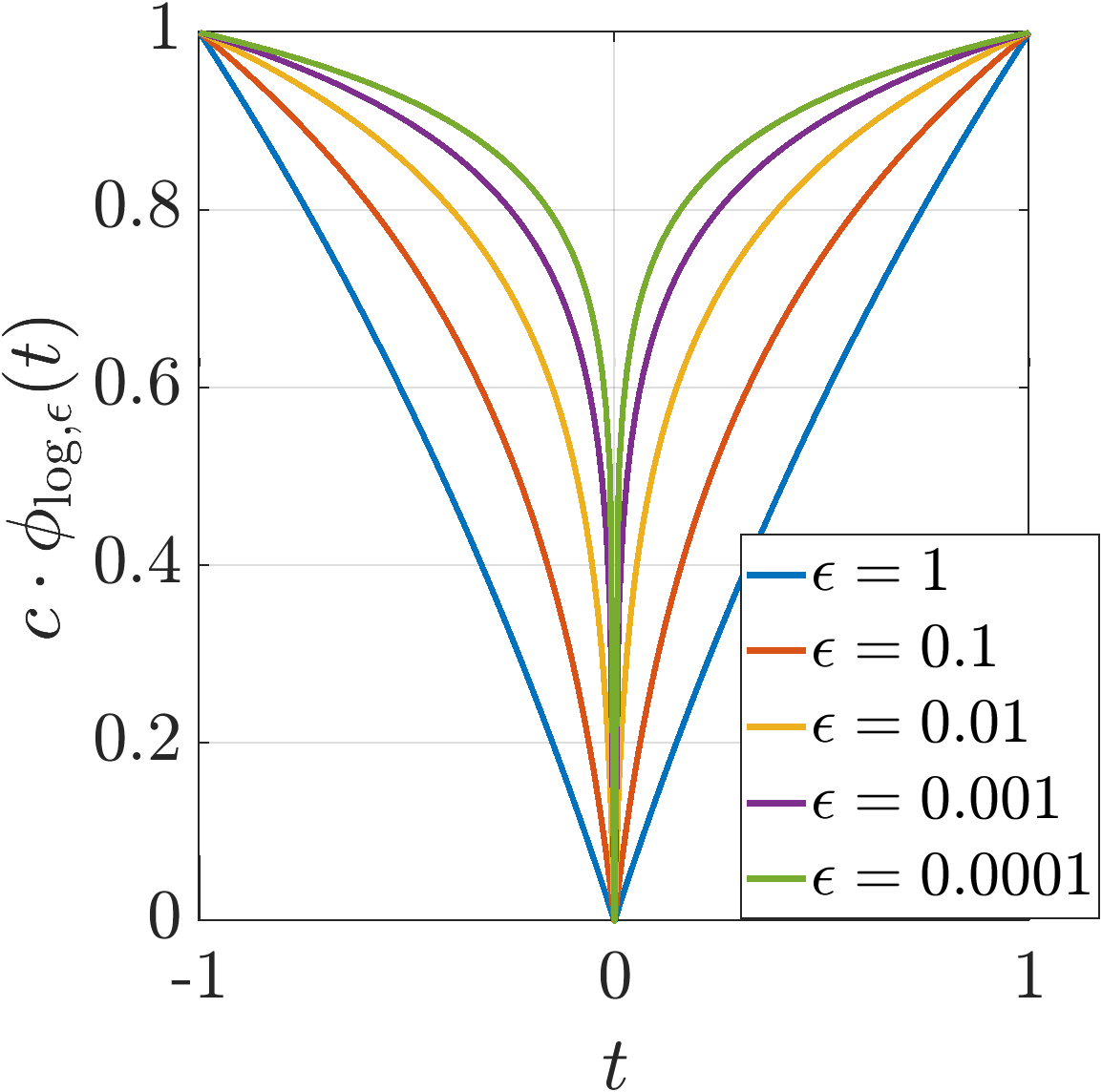}
    \end{subfigure}%
    \hfill
    \begin{subfigure}[c]{0.47\linewidth}
        \centering
        \includegraphics[width=\linewidth]{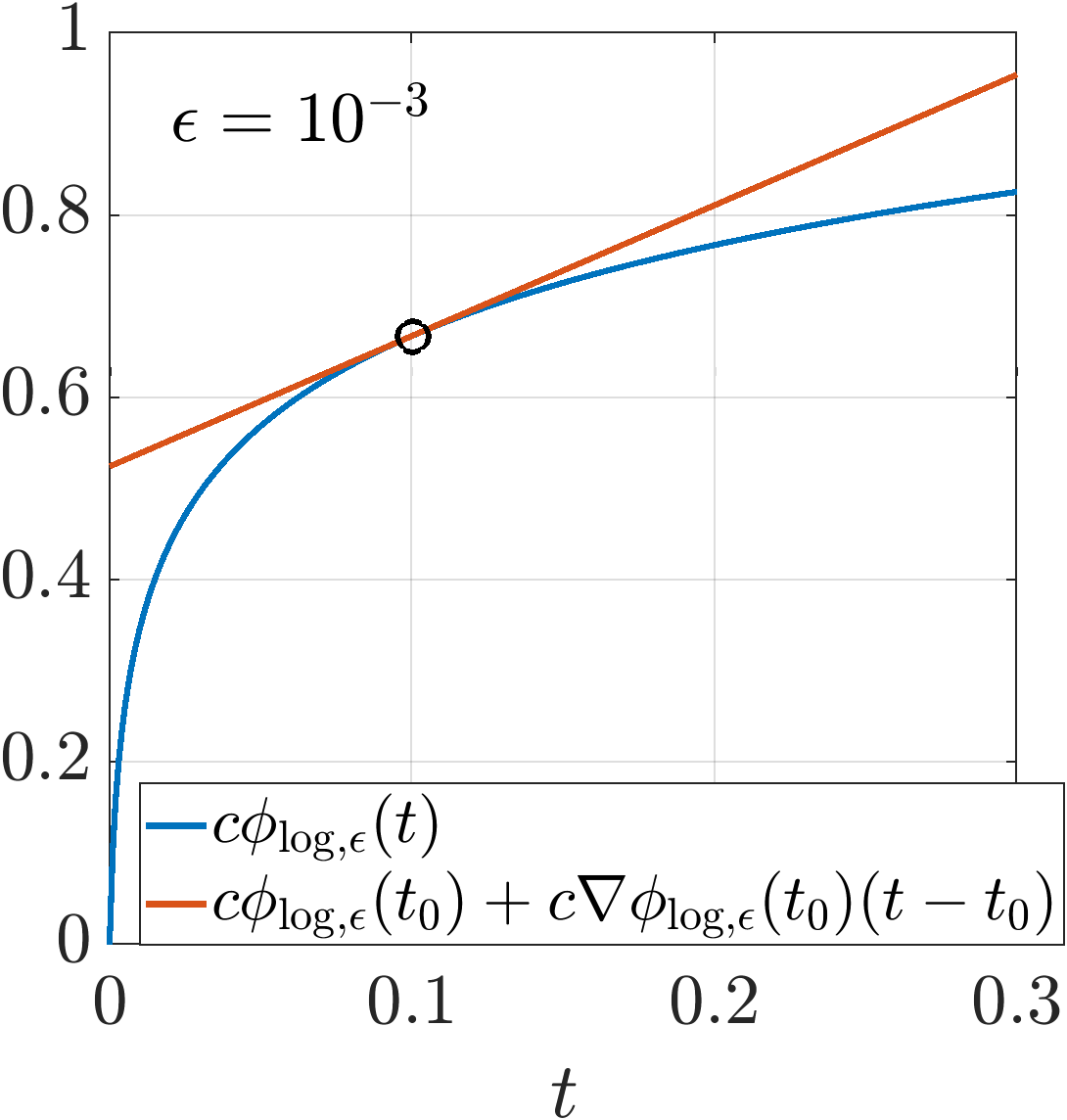}
    \end{subfigure}
    \caption{Left: $\phi_{\log,\epsilon}(t)$ approaches $\phi_0(t)$ as $\epsilon \rightarrow 0$.
    Right: For $t_0 > 0$, $\nabla \phi_{\log, \epsilon}(t_0) = 1 / (t_0 + \epsilon)$, which gives
    the weighted $\ell_1$ penalty, with its gradient becoming steeper as $t_0 \rightarrow 0$.}
    \label{fig:log_penalty_both}
\end{figure}

\section{Hands-Off Covariance Steering} \label{sec:sparse_covariance_feedback}

\subsection{Problem Formulation}
Define a binary vector $\tau \in \{0, 1\}^N$, where $\tau_k$ is such that
\begin{equation} \label{eq:tau_definition}
    \tau_k = \begin{cases}
        0 & \text{if } \P{u_k = 0} = 1 \\
        1 & \text{otherwise} .
    \end{cases}
\end{equation}
Hands-off covariance steering is a multi-objective optimization problem that trades off the transient cost and the sparsity of the feedback gains as follows:
\begin{equation} \label{eq:hands-off-covariance-steering}
    \min_{\pi, \tau} \quad [J_{\Sigma}, \ J_\tau := \sum_{k=0}^{N-1} \tau_k ]
    \quad \text{s.t.} \ \cref{eq:covariance_propagation_in_K}, \cref{eq:boundary_conditions}, \cref{eq:tau_definition}
\end{equation}
Our objective is to find, in a computationally tractable manner, an approximate Pareto front of the transient cost and sparsity.

When our control policy takes the form \cref{eq:control_policy}, we want 
the feedback gain matrices $(K_k)_{k=0}^{N-1}$ to be zero for many time intervals.
Since the feedback gains are related to the variables $(Y_k)_{k=0}^{N-1}$
through the relation \cref{eq:Y_relaxation}, we place a regularization term on $(Y_k)_{k=0}^{N-1}$.
We will show later that $Y_k = 0$ if and only if $K_k = 0$ from \cref{assum:nondegenerate_covariance}.

\subsection{Regularized Covariance Steering}
A natural first step to solve \cref{eq:hands-off-covariance-steering} is to add a regularization term to the objective function
and to vary the regularization parameter.
The regularized covariance steering problem is given by
\begin{equation} \label{eq:regularized_covariance_steering}
    \min_{\Sigma_k, U_k, Y_k} J_{\Sigma} + \lambda \sum_{k=0}^{N-1} \|Y_k\|_F 
    \text{ s.t. } \cref{eq:boundary_conditions_2}, \cref{eq:covariance_dynamics_in_Y}, \cref{eq:psd_constraint}
\end{equation}
where $\lambda > 0$ is the regularization parameter. 
Note that taking the Frobenius norm of $Y_k$ is equivalent to taking the $\ell_{2}$-norm of the vectorized form of $Y_k$.
Hence, this is a $\ell_{1,2}$ norm-regularized covariance steering problem.

However, as we demonstrate in \cref{sec:numerical_example}, solving the problem with this objective function does not always provide sparsity in the feedback gains.
As a method of inducing sparsity, we turn to the IRL1 method.

\subsection{IRL1P for Covariance Steering}
Sum-of-norms regularization and IRL1 are combined here as a way of achieving better sparsity in the solution.
We call this method the iteratively reweighted $\ell_{1,p}$-norm minimization (IRL1P).
The problem is given by
\begin{equation} \label{eq:IRL1p-covariance-steering}
    \hspace{-4pt}\min_{\Sigma_k, U_k, Y_k} J_{\Sigma} + \lambda \sum_{k=0}^{N-1} w_k^{(l)} \|Y_{k}\|_F 
    \ \text{ s.t. } \cref{eq:boundary_conditions_2}, \cref{eq:covariance_dynamics_in_Y},\cref{eq:psd_constraint} .
\end{equation}
The update rule for the weights is given by
\begin{equation}
    w_k^{(l+1)} = \frac{1}{\|Y_k^{(l)}\|_F + \epsilon}, \quad w_k^{(1)} = 1, \ k \in \Z_{0:N-1}
\end{equation}
and the algorithm terminates when 
\begin{equation}\label{eq:convergence_criterion}
    \frac{\sum_{k=0}^{N-1} (\norm{K_k}_F^{(l)} - \norm{K_k}_F^{(l-1)} )}{\sum_{k=0}^{N-1} \norm{K_k}_F^{(l-1)}} < \epsilon_{\mathrm{conv}}
\end{equation}
for some tolerance $\epsilon_{\mathrm{conv}} > 0$.
For completeness, the pseudocode is given in \cref{alg:IRL1-CS}.
\begin{algorithm}[thb]
\begin{algorithmic}[1]
\Require $\lambda, \epsilon, \epsilon_{\mathrm{conv}}, l_{\max}$
\State Initialize $w_k^{(1)} = 1, \ k \in \Z_{0:N-1}$
\While{$l < l_{\max}$}
\State Solve \cref{eq:IRL1p-covariance-steering} with $\lambda, w_k^{(l)}$ 
\State Retrieve feedback gains $K_k = U_k \Sigma_k^{-1}, \ k \in \Z_{0:N-1}$
\State If $l > 1$ and \cref{eq:convergence_criterion} are satisfied, \textbf{break}
\State Update weights $w_k^{(l+1)} = \frac{1}{\|Y_k^{(l)}\|_F + \epsilon}, \ k \in \Z_{0:N-1}$
\State $l \gets l+1$
\EndWhile
\State \textbf{return} $K_k, Y_k$ for $k \in \Z_{0:N-1}$, $\Sigma_k$ for $k \in \Z_{0:N}$
\end{algorithmic}
\caption{Iteratively Reweighted $\ell_{1,p}$-norm Minimization for Covariance Steering}
\label{alg:IRL1-CS}
\end{algorithm}

\begin{remark}
    The original IRL1 paper \cite{candesEnhancingSparsityReweighted2008} 
    provides no guarantee of convergence.
    A related work \cite{chenConvergenceReweightedL12014} shows that for problems of the form
    \begin{equation} \label{eq:IRL1-chen}
        \min_x \norm{Ax - b}_2^2 + \lambda \norm{x}_p^p, \quad 0 < p < 1
    \end{equation}
    replacing the regularization term with the weighted $\ell_1$-norm and applying the IRL1 procedure produces a 
    bounded sequence and any accumulation point is a stationary point of \cref{eq:IRL1-chen}. 
    At this point, we do not provide a convergence proof for \cref{alg:IRL1-CS}.
    However, as has been demonstrated in many other applications, the method is effective in
    both convergence and providing a sparse solution.
\end{remark}
\begin{assumption} \label{assum:slater_condition}
    Slater's condition holds for \cref{eq:IRL1p-covariance-steering}, i.e. there exists a strictly feasible solution for \cref{eq:IRL1p-covariance-steering}.
\end{assumption}
\subsection{Addition of Chance Constraints}
In addition to the standard formulation in \cref{eq:hands-off-covariance-steering}, we consider chance constraints on the control input's Euclidean norm:
\begin{equation} \label{eq:chance-constraint}
    \P{\norm{u_k}_2 \leq u_{\max}} \geq 1 - \gamma
\end{equation}
where $0 < \gamma < 1$ is a user-specified probability of violating the term inside $\P{\cdot}$.
The following lemma enables us to tractably reformulate \cref{eq:chance-constraint}
in a deterministic parameter optimization:
\begin{lemma}[Lemma 3, \cite{oguriChanceConstrainedControlSafe2024b}] \label{lemma:chance-constraint-upper-bound}
    For a Gaussian random vector $\xi \in \R^n$ with mean $\mu$ and covariance matrix $\Sigma$,
    the quantile function of $\norm{\xi}_2$, denoted by 
    $\mathcal{Q}_{\norm{\xi}_2}(p)$, is upper bounded as
    \begin{equation} \label{eq:chance-constraint-upper-bound}
        \mathcal{Q}_{\norm{\xi}_2}(p) \leq \norm{\mu}_2 + \sqrt{\mathcal{Q}_{\chi_n^2}(p) \lambda_{\max}(\Sigma)}
    \end{equation}
    where $\mathcal{Q}_{\chi_n^2}(p)$ is the $p$-quantile of the $\chi^2$ distribution with $n$ degrees of freedom.
\end{lemma}
From \cref{lemma:chance-constraint-upper-bound},
\cref{eq:chance-constraint} can be imposed as
\begin{equation}
    \E{u_k} + \sqrt{\mathcal{Q}_{\chi_m^2}(1-\gamma) \lambda_{\max}(\cov{u_k})} \leq u_{\max} .
\end{equation}
From \cref{assum:zero_mean}, $\E{u_k} = 0$ for $k \in \Z_{0:N-1}$. Moving terms around, 
\begin{equation} \label{eq:cc-reformulation-1}
    \lambda_{\max}(\cov{u_k}) \leq \frac{u_{\max}^2}{\mathcal{Q}_{\chi_n^2}(1-\gamma)} .
\end{equation}
Assuming that lossless convexification holds under chance constraints, i.e.
$Y_k = \cov{u_k}$, which we later prove in \cref{thm:lossless}, \cref{eq:cc-reformulation-1} can be rewritten as
\begin{equation} \label{eq:cc-reformulation-2}
    \lambda_{\max}(Y_k) \leq \rho := \frac{u_{\max}^2}{\mathcal{Q}_{\chi_n^2}(1-\gamma)}
\end{equation}
which can be readily included in \cref{eq:IRL1p-covariance-steering}, still as an SDP \cite{boydConvexOptimization2004}.
\begin{remark}
    Chance constraints are not restricted to this form or to be on the control input; 
    see e.g. \cite{oguriChanceConstrainedControlSafe2024b} for other
    deterministic reformulations of chance constraints, such as hyperplane constraints.
\end{remark}

\subsection{Lossless Convexification with Regularization}
Here, we show that the lossless convexification property shown in \cite{liuOptimalCovarianceSteering2024}
holds even with the regularization term in \cref{eq:IRL1p-covariance-steering}
and chance constraints in \cref{eq:cc-reformulation-2}.
Define $F_k, G_k, h_k$ as
\begin{align}
    F_k &:= A_k \Sigma_k A_k^\top + A_k U_k^\top B_k^\top 
    + B_k U_k A_k^\top \nonumber\\
    &\quad + B_k Y_k B_k^\top + D_k D_k^\top - \Sigma_{k+1} \\
    G_k &:= U_k \Sigma_k^{-1} U_k^\top - Y_k \\
    h_k &:= \lambda_{\max}(Y_k) - \rho .
\end{align}
Eq. \cref{eq:IRL1p-covariance-steering}, without boundary conditions, is now written as
\begin{subequations}
\begin{align}
    \min_{\Sigma_k, U_k, Y_k} &\sum_{k=0}^{N-1} \tr(Q_k \Sigma_k) + \tr(R_k Y_k) + \lambda w_k \|Y_k\|_F \\
    \text{s.t.} \ &F_k = 0, \ G_k \preceq 0, \ h_k \leq 0 .
\end{align}
\end{subequations}

The following lemma is used in the proof.
\begin{lemma}[Lemma 1, \cite{liuOptimalCovarianceSteering2024}] \label{lemma:AB}
    Let $A$ and $B$ be $n \times n$ symmetric matrices with $A \succeq 0$, $B \preceq 0$, and 
    $\tr(AB) = 0$. If $B$ has at least one nonzero eigenvalue, then $A$ is singular.
\end{lemma}

\begin{theorem} \label{thm:lossless}
    % Let Slater's condition \cite{boydConvexOptimization2004} hold for \cref{eq:IRL1p-covariance-steering}.
    At the optimal solution of \cref{eq:IRL1p-covariance-steering}, $G_k = 0$ for $k \in \Z_{0:N-1}$, hence the relaxation is lossless.
\end{theorem}

\begin{proof}

Let $\Lambda_k$, $\Gamma_k$, $\eta_k$ be the Lagrange multipliers for $F_k$, $G_k$, $h_k$ respectively.
$\Gamma_k$ is symmetric by definition, and $\Lambda_k$ is symmetric because of the symmetry in $F_k$.
\cref{assum:slater_condition} (Slater's condition) implies strong duality. Then, from strong duality, the Karush-Kuhn-Tucker (KKT) conditions are necessary (and sufficient) for optimality \cite{boydConvexOptimization2004}.
The Lagrangian is given by
\begin{align}
    &\mathcal{L}(\Sigma_k, U_k, Y_k, \Lambda_k, \Gamma_k) = 
    \sum_{k=0}^{N-1} \tr(Q_k \Sigma_k) + \tr(R_k Y_k) \nonumber \\
    &  \quad + \lambda w_k \|Y_k\|_F + \tr(\Lambda_k F_k) + \tr(\Gamma_k G_k) + \eta_k h_k .
\end{align}
The relevant KKT conditions are given by
\begin{subequations}
\begin{align}
    % &\frac{\partial \mathcal{L}}{\partial \Sigma_k} 
    % = Q_k + A_k^\top \Lambda_k A_k - \Sigma_k^{-1}U_k^\top \Gamma_k U_k \Sigma_k^{-1} - \Lambda_{k+1} = 0 \label{eq:dLdSigma} \\
    &\frac{\partial \mathcal{L}}{\partial U_k} 
    = 2 \Gamma_k U_k \Sigma_k^{-1} + 2 B_k^\top \Lambda_k A_k = 0 \label{eq:dLdU} \\
    &\frac{\partial \mathcal{L}}{\partial Y_k} 
    = R_k +  \frac{\lambda w_k Y_k}{\|Y_k\|_F} - \Gamma_k + \eta_k v_k v_k^\top + B_k^\top \Lambda_k B_k = 0 \label{eq:dLdY} \\
    &F_k = 0, \ G_k \preceq 0, \ h_k \leq  0\quad \text{(primary feasibility)} \\ 
    & \Gamma_k \succeq 0, \eta_k \geq 0  \quad \text{(dual feasibility)} \label{eq:dual_feasibility}  \\
    & \tr(\Gamma_k^\top G_k) = \eta_k h_k = 0 \ \ \text{(complementary slackness)} 
\end{align}
\end{subequations}
where $v_k$ is the eigenvector corresponding to the largest eigenvalue of $Y_k$.
From \cref{eq:dLdU}, we have
\begin{equation}
    B_k^\top \Lambda_k = - \Gamma_k U_k \Sigma_k^{-1} A_k^{-1} .
\end{equation}
Substituting into \cref{eq:dLdY}, we have
\begin{equation} 
    R_k +  \frac{\lambda w_k Y_k}{\|Y_k\|_F}  - \Gamma_k + \eta_k v_k v_k^\top - \Gamma_k U_k \Sigma_k^{-1}  A_k^{-1} B_k = 0 \label{eq:dLdY_substituted}
\end{equation}

Now, we are ready to show the theorem statement using proof by contradiction.
Assume that $G_k$ has at least one nonzero eigenvalue.
From $G_k \preceq 0$, symmetry of $\Gamma_k$, and $\tr(\Gamma_k^\top G_k) = 0$, applying \cref{lemma:AB} gives that $\Gamma_k$ is singular. 
Then, moving terms in \cref{eq:dLdY_substituted} and taking the determinant of both sides,
\begin{align}
&\det(R_k +  \frac{\lambda w_k Y_k}{\|Y_k\|_F} + \eta_k v_k v_k^\top)  \nonumber\\
&\quad = \det(\Gamma_k) \det(I + \Gamma_k U_k \Sigma_k^{-1}  A_k^{-1} B_k) .
\end{align}
The LHS is positive, since $R_k \succ 0, \lambda \geq 0, Y_k \succeq 0$ by assumption and $\eta_k \geq 0$ from \cref{eq:dual_feasibility}.
However, since $\Gamma_k$ is singular, its determinant, and hence the RHS is zero. Thus, we have a contradiction.
\end{proof}

\begin{remark}
    From \cref{assum:nondegenerate_covariance}, 
    $Y_k = K_k \Sigma_k K_k^\top = 0$ implies $K_k = 0$. Then, $K_k = 0 \iff Y_k = 0$.
    With \cref{thm:lossless}, this justifies the use of the regularization on $Y_k$ in \cref{eq:IRL1p-covariance-steering}.
\end{remark}

\section{Numerical Example} \label{sec:numerical_example}
Here, we show the solution obtained from hands-off covariance steering. 
SDPs are solved using YALMIP \cite{lofbergYALMIPToolboxModeling2004} and MOSEK \cite{mosekapsMOSEKOptimizationToolbox2023}.
We use some problem parameters from \cite{liuOptimalCovarianceSteering2024}: for $k \in \Z_{0:N-1}$, $Q_k = 0.5 I_2$, $R_k = 1$,
\begin{equation} \nonumber
    A_k = \begin{bmatrix}
        1 & 0.2 \\ 0 & 1
    \end{bmatrix}, \ B_k = \begin{bmatrix}
        0.02 \\ 0.2
    \end{bmatrix}, \ D_k = \begin{bmatrix}
        0.4 & 0 \\ 0.4 & 0.6
    \end{bmatrix} .
\end{equation}
We set the boundary covariance matrices as
\begin{equation} \nonumber
    \bar{\Sigma}_0 = \begin{bmatrix}
        5 & -1 \\ -1 & 1
    \end{bmatrix}, \quad 
    \bar{\Sigma}_N = \begin{bmatrix}
        0.5 & -0.4 \\ -0.4 & 2
    \end{bmatrix} .
\end{equation}
The chance constraint is defined by $u_{\max} = 10$, $\gamma = 0.03$.
We also note that for every problem solved, the lossless convexification
was verified to hold (to some numerical tolerance). 
Simulations were performed on a standard laptop with an Intel i7 processor and 16GB of RAM.

\subsection{Comparison with Brute-Force Search}
To first demonstrate the capability of the IRL1P method 
to find a sparse solution that closely matches the global optimum, 
we compare the solution obtained from the method
with the solution obtained from a brute-force search.
The brute-force search is performed by solving $2^N$ convex optimization problems, 
where each problem imposes $Y_k = 0$ for a set of nodes $k$.
The solution with the smallest objective function value is chosen.
To keep the search space small, we consider the problem with $N = 8$.
Furthermore, to emphasize the effectiveness for a wide range of sparsity, we exclude the 
chance constraints from the problem.

\begin{figure}[hbt]
    \centering
    \includegraphics[width=0.6\linewidth]{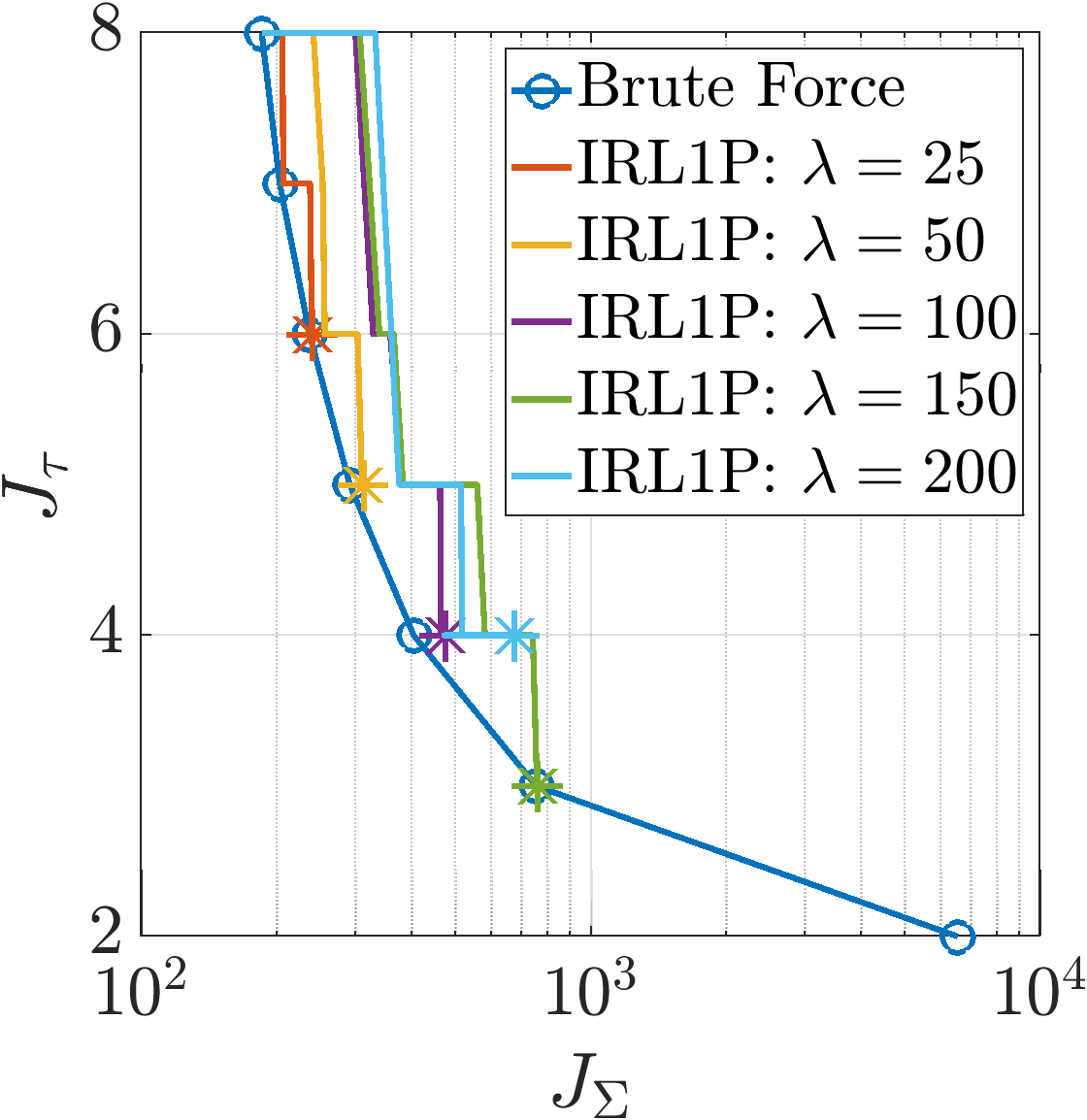}
    \caption{Comparison with brute-force search Pareto front. $*$ shows the final solution from IRL1P.}
    \label{fig:brute_force_vs_IRL1}
\end{figure}

\cref{fig:brute_force_vs_IRL1} shows the comparison.
Overall, increasing $\lambda$ leads to a more sparse solution, i.e. smaller $J_\tau$.
The final solution shows convergence to a pair $(J_{\Sigma}, J_\tau)$ that is close to the Pareto front obtained from the brute-force search.
For example, using $\lambda = 25$ converges to a solution with $J_{\tau} = 6$
that is close to brute-force.
The same can be said for $\lambda = 50$ to $J_{\tau} = 5$, 
$\lambda = 100$ to $J_{\tau} = 4$, and $\lambda = 150$ to $J_{\tau} = 3$.
However, the solution from IRL1P fails to find a solution with the desired sparsity for $J_{\tau} = 2$, 
most likely because this solution requires a large transient cost, and as a result many iterations.
Nevertheless, the solution from IRL1P finds a near-optimal solution with a much smaller computational cost, and its
effectiveness scales to problems with larger $N$. 
Brute-force search is not feasible for larger $N$ due to the exponential growth of the search space, 
while the computational complexity of IRL1P is proportional to the computational
complexity of solving the SDP. The computational effort of the SDP has been demonstrated to grow approximately linearly
with $N$ up to problem sizes of $N$ in the hundreds \cite{rapakouliasDiscreteTimeOptimalCovariance2023a}.
Cases where $J_\tau < 2$ were infeasible for this problem setting.
Since throughout the iteration progress, the solution from IRL1P is close to the Pareto front, this also
suggests that the user can also choose to terminate the algorithm when the desired sparsity is achieved, 
if they are to accept some suboptimality.

\subsection{Standard Covariance Steering}
The analysis from here on is performed with $N=29$ and the chance constraints.
The covariance evolution without feedback control and the covariance evolution
from solving \cref{eq:covariance_steering_in_U} are shown in \cref{fig:covariances_no_control}
and \cref{fig:covariances_original}, respectively.
In the standard covariance steering case, the final covariance constraint is active, i.e. $\Sigma_N = \bar{\Sigma}_N$.
\begin{figure}[htb]
    \centering
    \begin{subfigure}[c]{0.5\linewidth}
        \centering
        \includegraphics[width=\linewidth]{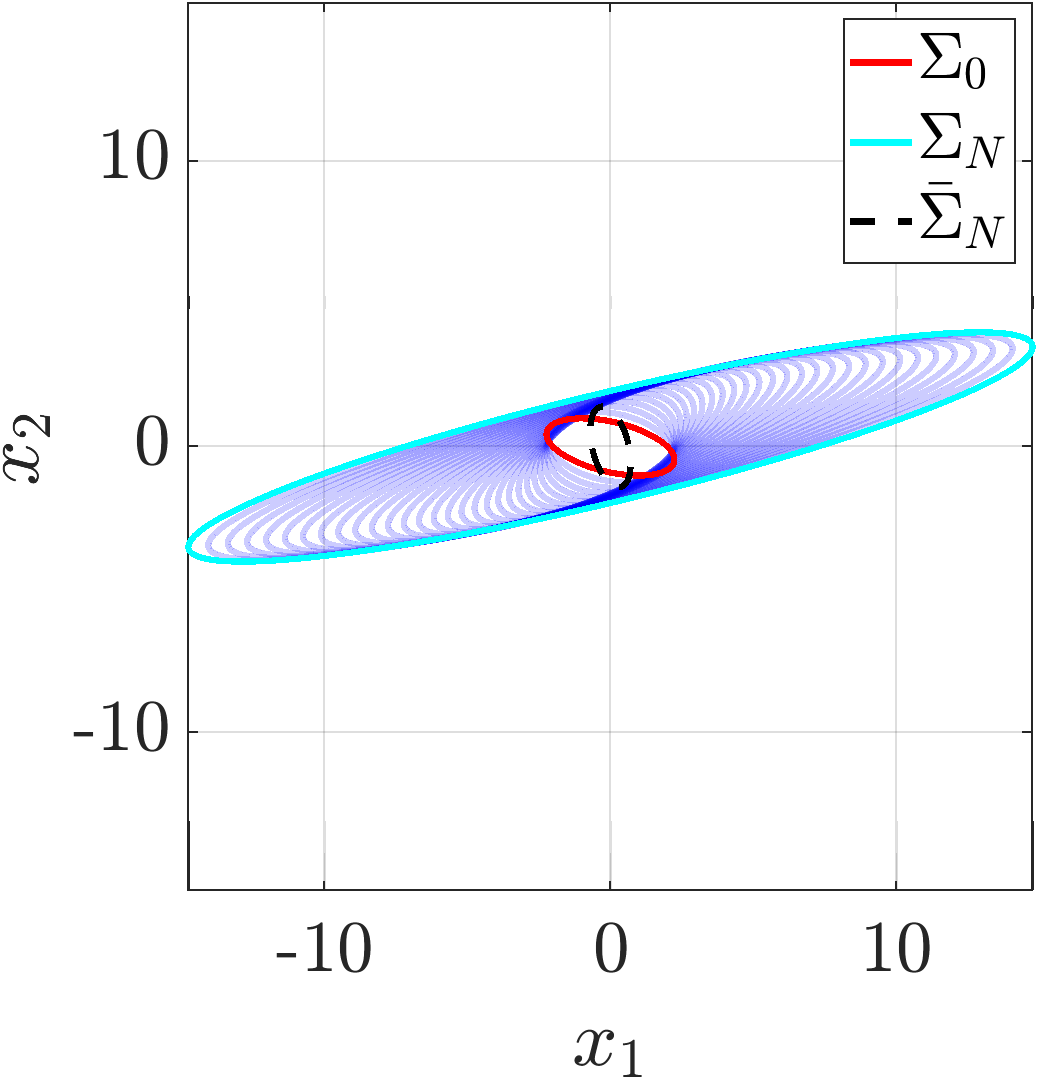}
        \caption{Without feedback control}
        \label{fig:covariances_no_control}
    \end{subfigure}%
    \begin{subfigure}[c]{0.5\linewidth}
        \centering
        \includegraphics[width=\linewidth]{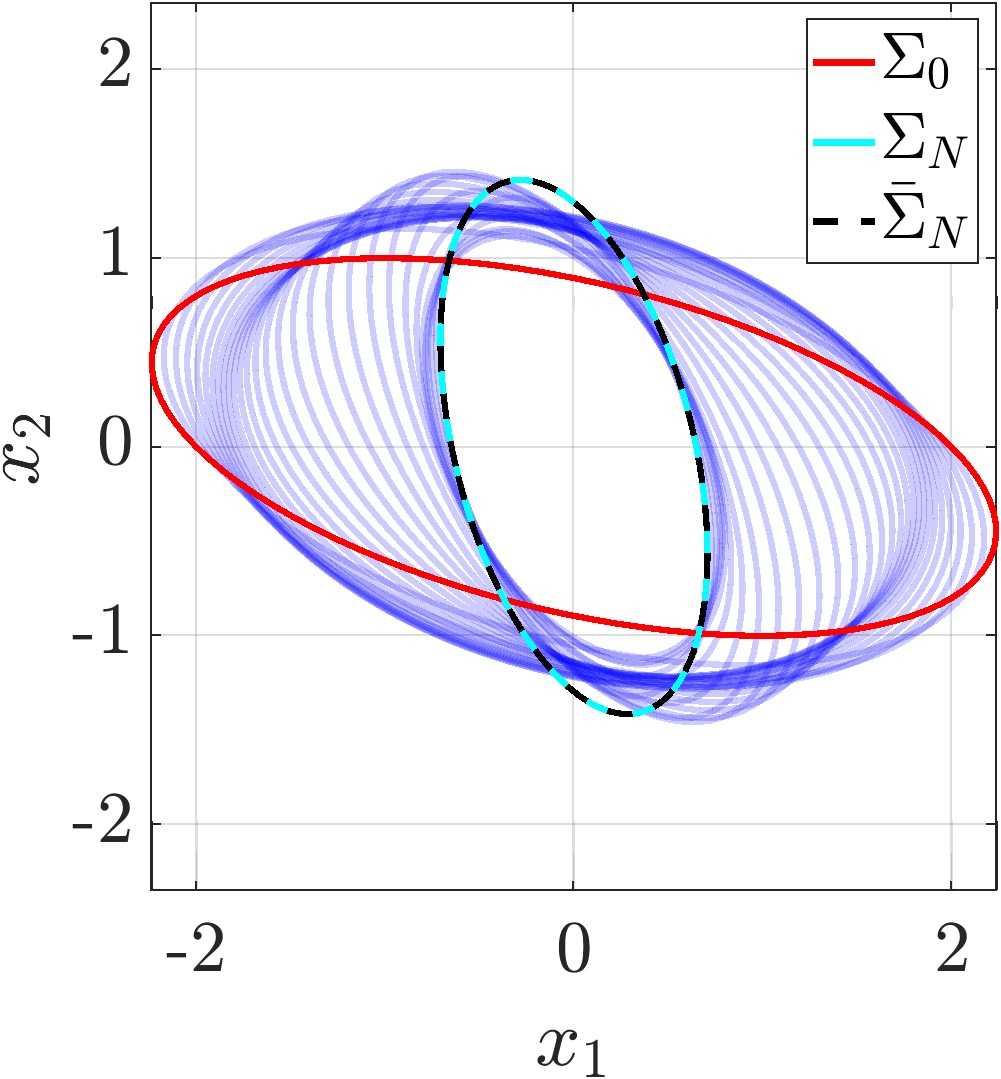}
        \caption{Standard covariance steering}
        \label{fig:covariances_original}
    \end{subfigure}
    \caption{1-sigma ellipses of the covariance evolution}
\end{figure}

\subsection{Iteratively Reweighted $\ell_{1,p}$ Regularization}
Next, we apply the proposed method, i.e. \cref{alg:IRL1-CS}. 
We use the parameters $\epsilon = 10^{-3}$, $l_{\max} = 50$.
\cref{fig:Y_lambda_max} shows the value of $\lambda_{\max}(Y_k)$ at each iteration, 
for cases with/without chance constraints, when using $\lambda = 1000$.
For the case with chance constraints, MOSEK's solution time reported by YALMIP was 0.013 seconds per iteration on average, 
and the algorithm terminated after 25 iterations. (Note that YALMIP's overhead is much larger, 
and including this, it took 0.40 seconds on average.)
\begin{figure}[htb]
    \centering
    \begin{subfigure}[c]{0.5\linewidth}
        \centering
        \includegraphics[width=\linewidth]{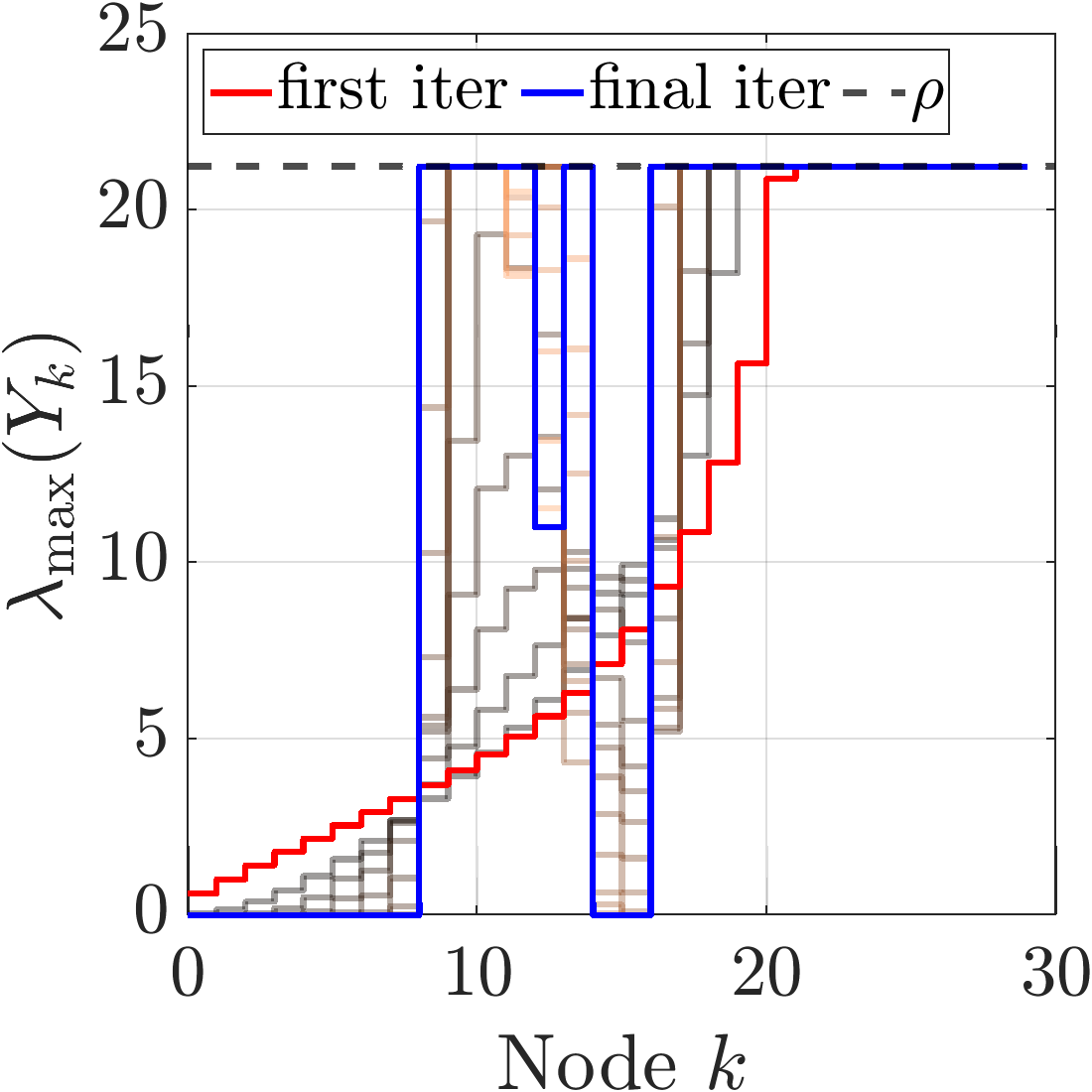}
        \caption{with chance constraints}
        \label{fig:Y_lambda_max_constrained}
    \end{subfigure}%
    \begin{subfigure}[c]{0.5\linewidth}
        \centering
        \includegraphics[width=\linewidth]{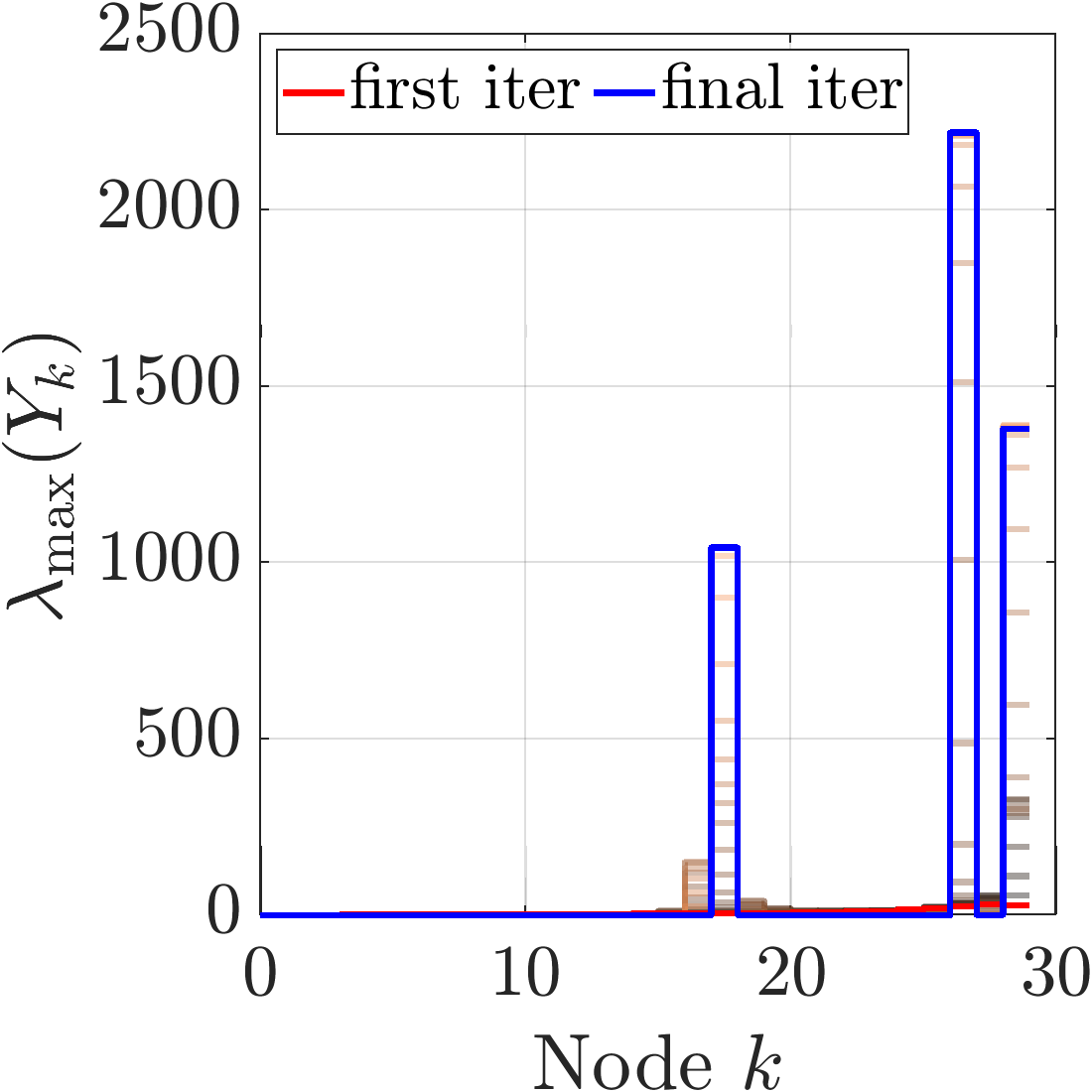}
        \caption{w/o chance constraints}
        \label{fig:Y_lambda_max_unconstrained}
    \end{subfigure}%
    \caption{Control covariance max eigenvalue at each iteration for $\lambda = 1000$. Darker colors correspond to 
    earlier iterations.}
    \label{fig:Y_lambda_max}
\end{figure}

Focusing first on \cref{fig:Y_lambda_max_constrained}, we can see that as the iteration progresses,
more nodes have $\lambda_{\max}(Y_k) \approx 0$, i.e. the feedback gain matrices are zero.
We can also see the effect of the chance constraints, as the maximum eigenvalue of $Y_k$ is bounded.
We remark that the converged solution shows bang-bang profile for $\lambda_{\max}(Y_k)$. 
This suggests that the equivalence of $\ell_1$-optimal control and sparse control shown in \cite{nagaharaMaximumHandsOffControl2016} 
extends to the control of the state covariance when the input covariance is bounded.
Without chance constraints, the algorithm converges to a solution 
that uses only 3 statistically large inputs, as shown in \cref{fig:Y_lambda_max_unconstrained}.
Since the chance constraints are not present, these nodes have much larger values of $\lambda_{\max}(Y_k)$ compared to the constrained case.
\begin{figure}[htb]
    \centering
    \includegraphics[width=0.5\linewidth]{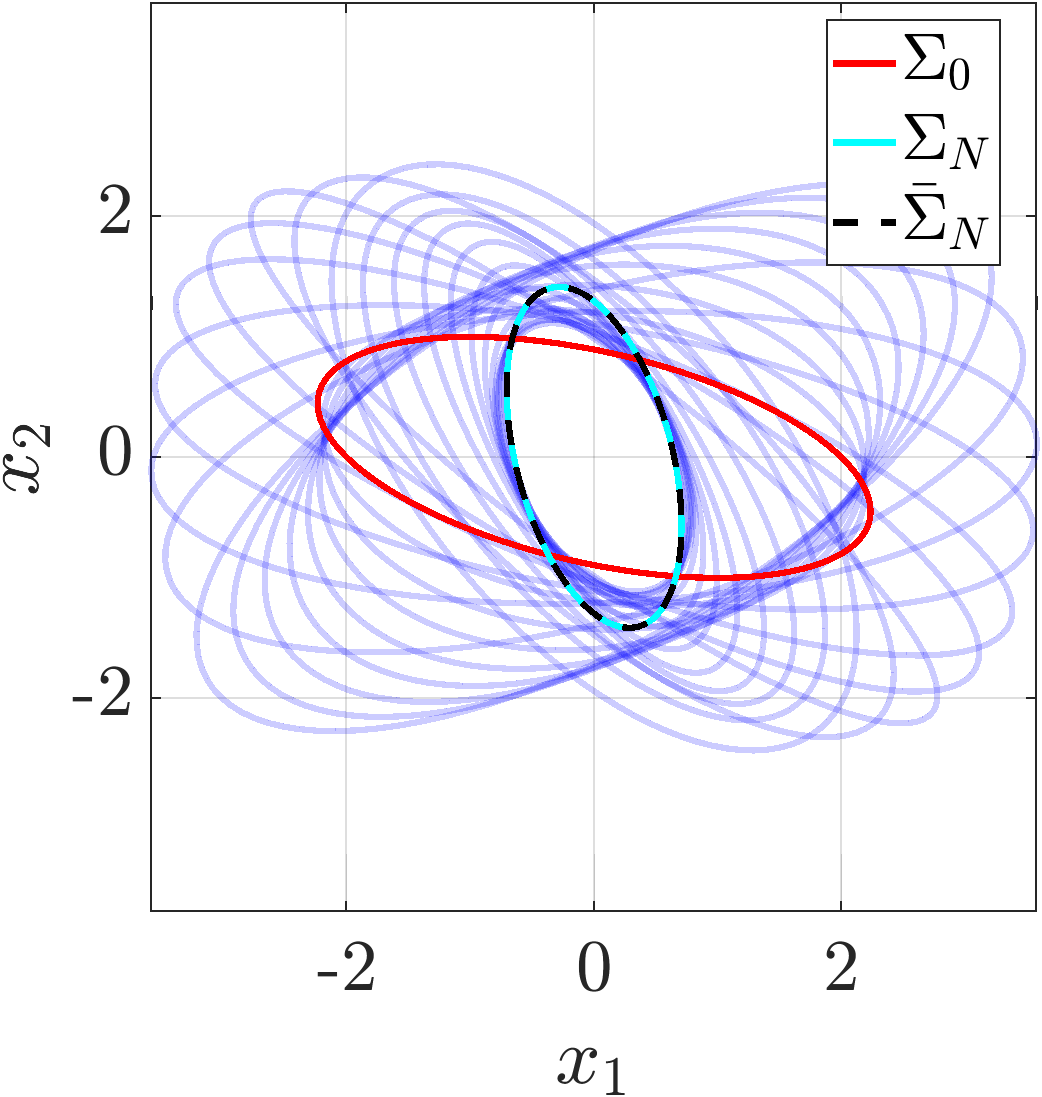}
    \caption{Covariance 1-sigma evolution for IRL1P, $\lambda = 1000$}
    \label{fig:covariance_IRL1}
\end{figure}

\cref{fig:covariance_IRL1} shows the state covariance evolution. 
While the covariance ellipse becomes larger compared to the standard covariance steering solution in \cref{fig:covariances_original}, the final covariance is steered to match the target.

\begin{figure}[htb]
    \centering
    \begin{subfigure}[c]{0.41\linewidth}
        \centering
        \includegraphics[width=\linewidth]{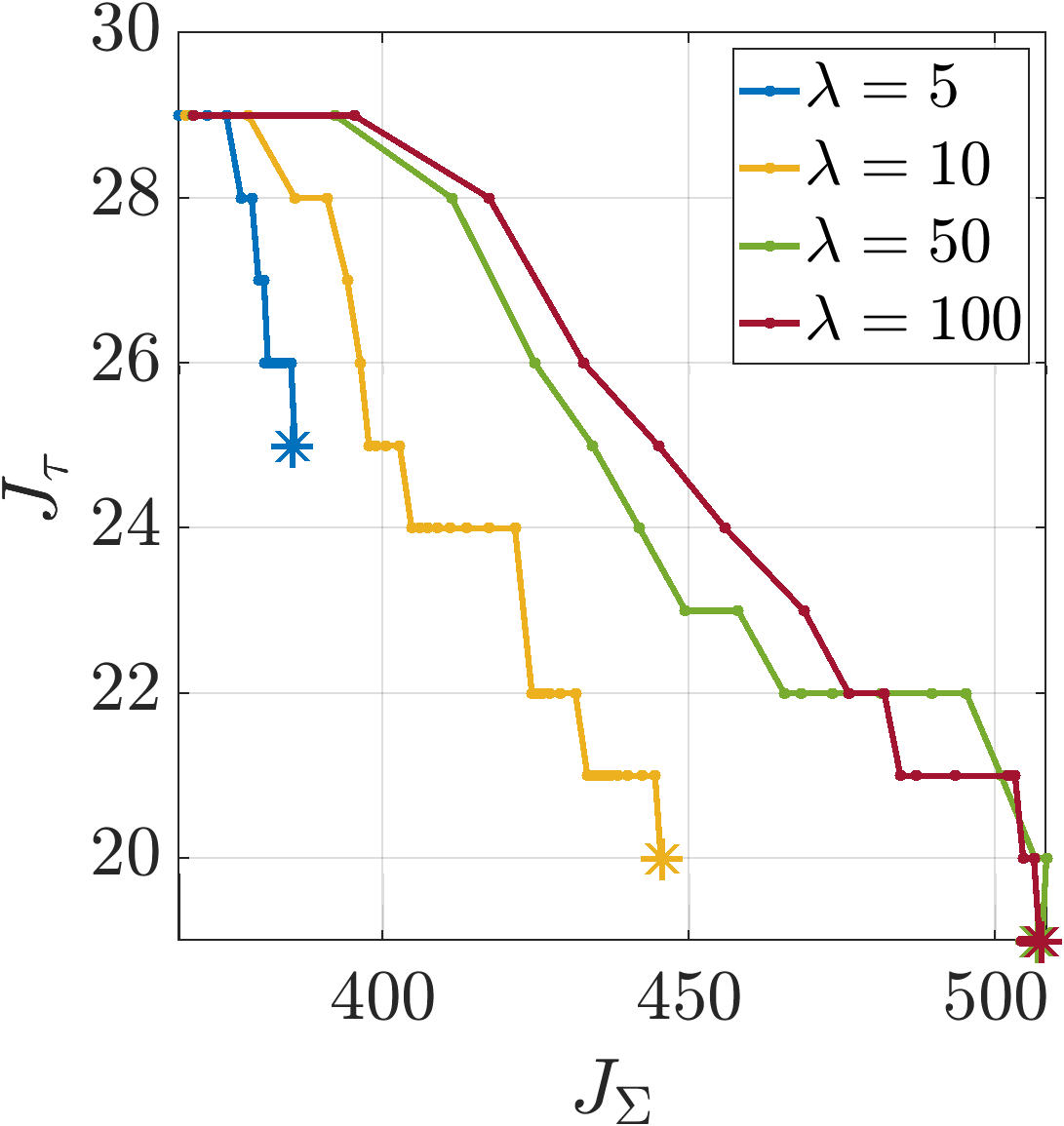}
        \caption{}
        \label{fig:lambda_variations}   
    \end{subfigure}%
    \begin{subfigure}[c]{0.59\linewidth}
        \centering
        \includegraphics[width=\linewidth]{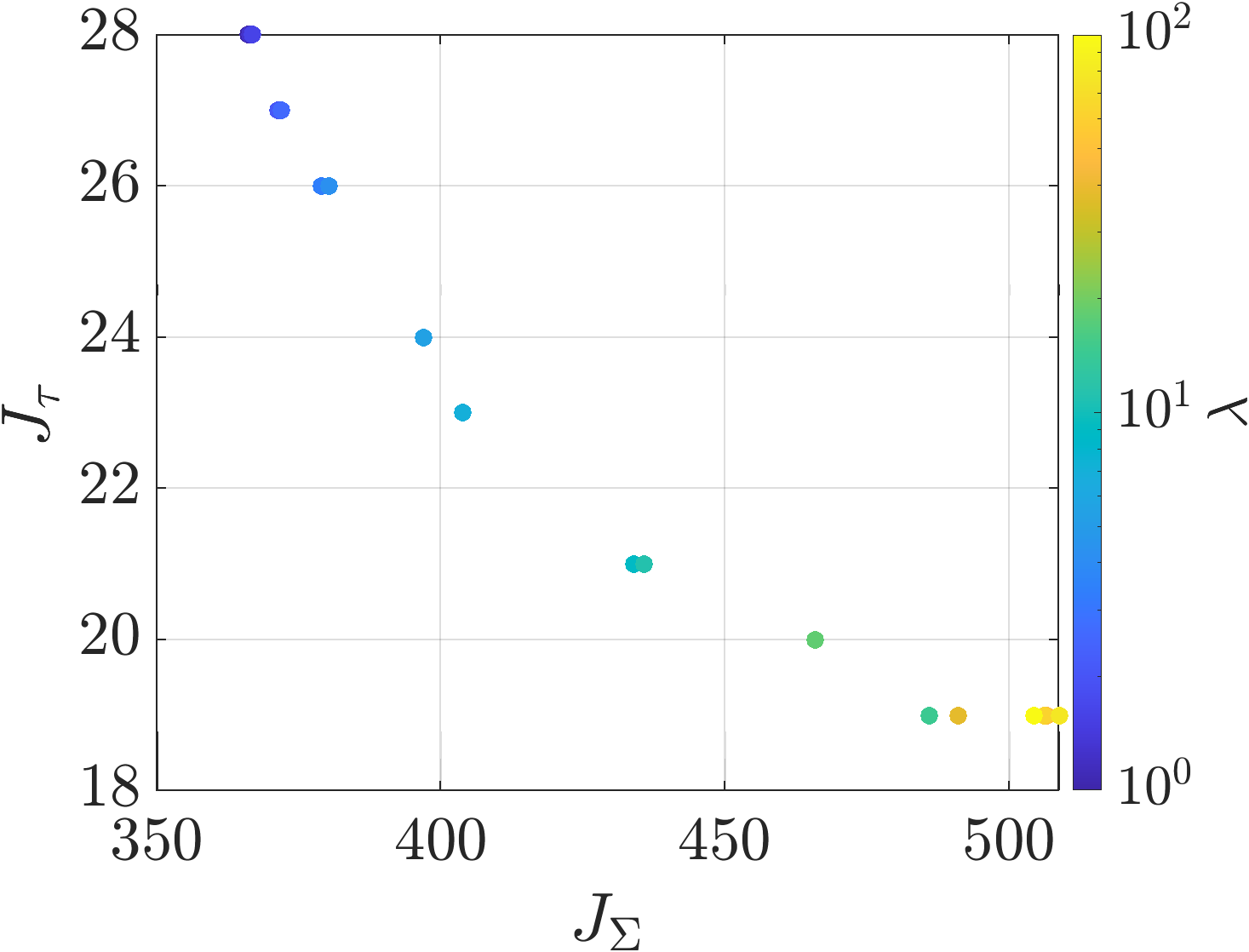}
        \caption{}
        \label{fig:transcient_vs_sparsity}
    \end{subfigure}
    \caption{Left: Algorithm behavior for different $\lambda$. The algorithm starts at the top-left and $*$ shows the final solution.
    Right: $J_{\Sigma}$ vs $J_{\tau}$ at algorithm termination for different $\lambda$.}
    \label{fig:trade_studies}
\end{figure}
\cref{fig:lambda_variations} shows the transient cost $J_{\Sigma}$
vs the number of nonzero feedback gain matrices $J_{\tau}$ for several values of $\lambda$.
As the iteration progresses, the transient cost increases, while 
sparsity is improved, i.e. the algorithm traverses the line in \cref{fig:lambda_variations} from 
top-left to bottom-right.
For sufficiently large $\lambda$, (in this case $\lambda= 50, 100$),
the algorithm behavior becomes similar, and increasing $\lambda$ further does not improve sparsity.
% For this case, $J_{\Sigma}$ increases without $J_{\tau}$ decreasing.
% For larger $\lambda$, the first few iterations only increase the transient cost, 
% but after this initial phase, most subsequent iterations improve the sparsity of the solution.
% This initial phase is where the unweighted regularization gets stuck and cannot induce sparsity, and shows the benefit of the iteratively reweighting scheme.

\cref{fig:transcient_vs_sparsity} shows the $J_{\Sigma}$-$J_{\tau}$ plot at the end of the algorithm for different values of $\lambda$.
Here, 20 values of $\lambda$ are chosen to be logarithmically spaced between $1$ and $100$.
The plot clearly shows the trade-off between the two objectives
based on the value of $\lambda$.
% For larger $\lambda$, $J_{\Sigma}$ is larger, while $J_{\tau}$ is smaller.
% We can observe that the final value does not change much for larger $\lambda$.
% From the two figures in \cref{fig:trade_studies}, we can conclude that in order to obtain the sparsest solution, 
% we should choose a large $\lambda$. 
% If a user desires a solution with more balance 
% between the two objectives, they can still use the same large $\lambda$ and terminate the algorithm early.

\subsection{Ineffectivity of Simple Regularization}
Here, we show the result from solving \cref{eq:regularized_covariance_steering}, i.e. unweighted regularization.
The resulting $\lambda_{\max}(Y_k)$ for each $k$ is shown in \cref{fig:K_frobenius_regularized_vs_original}. This example tested $\lambda = 10^4, 10^{10}$.
For comparison, values from standard covariance steering are also shown.
We can see that simple regularization does not induce sparsity 
in the control input covariance, even for large $\lambda$.
This matches the result in \cref{fig:lambda_variations} for IRL1P, where 
the first iterations of the algorithm do not improve the sparsity.
\begin{figure}[htb]
    \centering
    \includegraphics[width=0.6\linewidth]{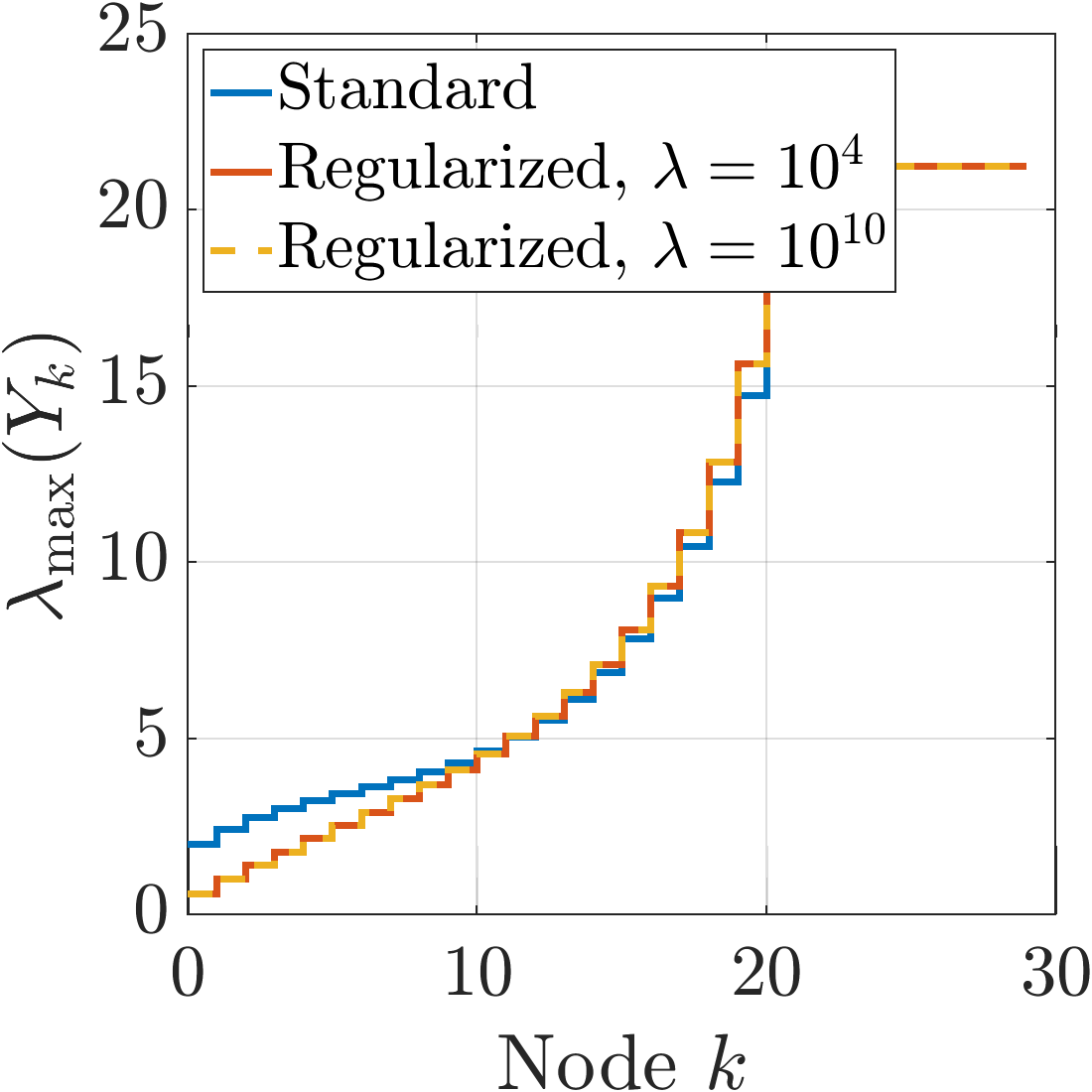}
    \caption{$\lambda_{\max}(Y_k)$, standard/regularized covariance steering}
    \label{fig:K_frobenius_regularized_vs_original}
\end{figure}

\section{Conclusions}
We demonstrate chance-constrained, hands-off covariance steering by applying iteratively reweighted $\ell_{1,p}$-norm minimization 
as a means of approximately solving $\ell_0$-norm regularization in the feedback gain matrices.
We prove that the lossless convexification property of covariance steering holds under the additional regularization term.
The method efficiently explores the trade between transient cost and sparsity in the feedback gain matrices
by solving a series of convex optimization problems.
Future works include application to path planning under chance constraints and nonlinear dynamics.

\balance % balance the columns

\bibliographystyle{ieeetr}
\bibliography{cdc2025.bib}

\end{document}